\documentclass[a4paper]{article}
\usepackage[latin1]{inputenc}
\usepackage{fullpage}
\usepackage{quaternions} 
\usepackage{amsthm}
\usepackage{xspace,tabularx,url,verbatim,color}
\usepackage[pdftex,colorlinks=true]{hyperref} 
\newcommand{\matlab}{\textsc{matlab}\textregistered\xspace}
\newtheorem{lemma}{Lemma}
\newtheorem{theorem}{Theorem}

\newtheorem{proposition}{Proposition}
\title{Fundamental representations and
       algebraic properties of biquaternions or complexified quaternions}
\author{Stephen J. Sangwine\thanks{This paper was started in 2005 at the
                                   Laboratoire des Images et des Signaux (now part of the GIPSA-Lab),
                                   Grenoble, France with financial support from the Royal Academy of
                                   Engineering of the United Kingdom and the Centre National de la Recherche
                                   Scientifique (CNRS). Thanks are due to Dr Sebastian Miron, who gave a
                                   short series of seminars on biquaternions in Grenoble in 2005,
                                   and thus started us on the road to this paper. We also acknowledge the
                                   contribution made by Daniel Alfsmann who, working with one of us, and
                                   using his knowledge of hypercomplex algebras in general, greatly clarified
                                   the topic of divisors of zero \cite{arXiv:0812:1102,10.1007/s00006-xxxx-xxxx-x}.}\\
        School of Computer Science and Electronic Engineering,\\
        University of Essex, Wivenhoe Park,\\
        Colchester, CO4 3SQ, United Kingdom.\\
        Email:~\href{mailto:s.sangwine@ieee.org}{S.Sangwine@IEEE.org}
\and
        Todd A. Ell\thanks{Dr T. A. Ell is a Visiting Fellow at the University of Essex, funded by grant
                           numbers GR/S58621 and EP/E010334/1 from the
                           United Kingdom Engineering and Physical Sciences Research Council from
                           September 2003 -- September 2009.}\\
        5620 Oak View Court, Savage, MN 55378-4695, USA.\\
        Email:~\href{mailto:t.ell@ieee.org}{T.Ell@IEEE.org}\\
\and
        Nicolas Le Bihan\\
        GIPSA-Lab\\
        Département Images et Signal\\
        961 Rue de la Houille Blanche, Domaine Universitaire\\
        BP 46, 38402 Saint Martin d'Hères cedex, France.\\
        Email:~\href{mailto:nicolas.le-bihan@gipsa-lab.inpg.fr}{nicolas.le-bihan@gipsa-lab.inpg.fr}\\
}
\begin{document}
\maketitle
\begin{abstract}
The fundamental properties of biquaternions (complexified quaternions) are presented
including several different representations,
some of them new,
and definitions of fundamental operations such
as the scalar and vector parts,
conjugates, semi-norms, polar forms,
and inner and outer products.
The notation is consistent throughout, even between representations,
providing a clear account of the many ways in which the component parts
of a biquaternion may be manipulated algebraically.
\end{abstract}
\section{Introduction}
\begin{quote}
It is typical of quaternion formulae that, though they be difficult to find,
once found they are immediately verifiable.\\\hspace*{\fill} J. L. Synge (1972) \cite[p34]{Synge:1972}
\end{quote}
The quaternions are relatively well-known but the quaternions with complex components (complexified
quaternions, or biquaternions\footnote{Biquaternions was a word coined by Hamilton himself
\cite{Hamiltonpapers:V3:35} and \cite[§\,669, p664]{Hamilton:1853}. (The word was used 18 years
later by Clifford \cite{10.1112/plms/s1-4.1.381} for a different concept, which is unfortunate.)})
are less so.
This paper aims to set out the fundamental definitions of biquaternions and some
elementary results, which, although elementary, are often not trivial.
The emphasis in this paper is on the biquaternions as an applied algebra -- that is, a
tool for the manipulation of algebraic expressions and formulae to allow deep insights into
scientific and engineering problems.
It is not a study of the abstract properties of the biquaternion algebra, nor
its relations with other algebras. Throughout the paper `quaternion' means a quaternion with
real components (a quaternion over the reals, \R), and `biquaternion' means a quaternion with
complex components (a quaternion over the field of complex numbers, \C). We denote the set of
quaternions by \H, and the set of biquaternions by \B. \H is, of course, a subset of \B.

Some of the material in the paper is based on the book by Ward \cite[Chapter 3]{Ward:1997}
which is one of the few sources of detail on the biquaternions.
We have not followed Ward's notation in this paper, preferring instead a
scheme based on bold or plain symbols without hats and underlines.

We have also drawn upon the paper by Sangwine and Alfsmann
\cite{arXiv:0812:1102,10.1007/s00006-xxxx-xxxx-x}
which sets out comprehensive results on the divisors of zero,
and their subsets the idempotents and nilpotents.
Sangwine and Alfsmann's paper uses the same notations as this paper
(having benefited from access to this paper in draft).

The quaternions themselves (with real elements) are well-covered in various books, for example
\cite{Altmann:1986,Artmann:88,Kantor:1989,KellandTait,Kuipers:1999}.
Hamilton's works on quaternions were published in book form in
\cite{Hamilton:1853,Hamilton:1866,Hamiltonpapers:V3},
and many are also now available freely on the Internet in various digital repositories.
The paper by Coxeter \cite{Coxeter:1946} is also a useful source.

We begin by setting out various ways in which quaternions and biquaternions may be represented.
We start with representations for quaternions,
for reference,
and because these representations are generalized for biquaternions.
In what follows, we use notation as consistently as we can.
In particular, the complex operator usually denoted by $i$
(or in electrical engineering $j$)
is represented in this paper as $\I$ in every case.
This is to keep the complex operator distinct from the first
of the three quaternion operators \i,
since it is independent.
The independence of \i and \I is perhaps the most fundamental
axiomatic aspect of the biquaternions that must be understood.
Bold symbols denote vectors and bivectors\footnote{Bivectors represent directed areas, and
are explained in Table~\ref{correspondence} and §\,\ref{sec:geometric}.},
whereas normal weight symbols denote scalar or complex quantities or quaternions.

Throughout the paper\footnote{Except in equation \ref{eqn:normprops}, where we use
the notation of the cited reference.}
we use the term \emph{norm} or the more specialized term \emph{semi-norm},
both denoted $\norm{q}$,
to mean the sum of the squares of the components of a quaternion or biquaternion,
and \emph{modulus}, denoted $\modulus{q}$,
to mean the square root of the norm and thus the Euclidean magnitude.
This is not universally accepted terminology,
many sources using \emph{norm} where we use \emph{modulus}.
However, our usage is consistent with several authors who have written on quaternions,
including Synge \cite{Synge:1972} and Ward \cite{Ward:1997},
but it does require care when using statements made about norms in other sources.

Many of the concepts given in this paper are implemented in numerical form in a
\matlab toolbox \cite{qtfm} which two of the authors first developed in 2005 and
are built upon in a toolbox for handling linear quaternion systems, first developed
in 2007 \cite{lqstfm}.
The toolbox \cite{qtfm} was essential to the development of this paper and the
results presented within it, since otherwise, errors in algebra would go unnoticed.
In many cases we have established results first by using the toolbox, and then derived
the algebraic proofs or statements which appear here.

\section{Quaternions}
\label{sec:realq}
Classically, quaternions are represented in the form of hypercomplex numbers with three imaginary
components. In Cartesian form this is:
\begin{equation}\label{classiccartesian}
q = w + x\i + y\j+ z\k
\end{equation}
where $\i$, $\j$ and $\k$ are mutually perpendicular unit bivectors\footnote{Classic texts often
refer to the operators $\i$, $\j$ and $\k$ as \emph{vectors},
a misconception that has caused considerable confusion over many years,
but is understandable, since it could not be cleared up without the
concept of geometric algebra and bivectors. We discuss this in §\,\ref{sec:geometric}.} obeying the
famous multiplication rules: $\i^2 = \j^2 = \k^2 = \i\j\k = -1$, discovered by Hamilton in 1843
\cite{Hamiltonpapers:V3:5}\nocite{Hamilton:1844}, and $w$, $x$, $y$, $z$, are real.
Quaternions are generalized to biquaternions by
permitting $w$, $x$, $y$ and $z$ to be complex, as discussed in §\,\ref{sec:complexq}, but
in this paper we reserve the symbols $w$, $x$, $y$, $z$ for the real case.
A quaternion with $w=0$ is known as a \emph{pure} quaternion (Hamilton's terminology, but still
used and widely understood).

The conjugate of a quaternion is given by negating the three imaginary components:
$\qconjugate{q} = w - x\i - y\j - z\k$. It is easily shown that
$\qconjugate{q}\,\qconjugate{p} = \qconjugate{p q}$ for general quaternions $p$ and $q$. Indeed
the formula may be generalized to more than two quaternions (the generalized formula was first noted
by Hamilton \cite[§\,20, p238]{Hamiltonpapers:V3:8}, and was also included in
\cite[p60, §\,31]{KellandTait}):
$\qconjugate{pqrst}=\qconjugate{t}\,\qconjugate{s}\,\qconjugate{r}\,\qconjugate{q}\,\qconjugate{p}$.
The quaternion conjugate may be expressed in terms of multiplications and additions
\cite[Theorem 11]{10.1016/j.camwa.2006.10.029} using any system of three mutually
orthogonal unit pure quaternions (here $\i,\j,\k$):
\[
\qconjugate{q} = -\frac{1}{2}\left(q+\i q\i + \j q\j + \k q\k \right)
\]
Similar formulae, based on involutions \cite{arXiv:math.RA/0506034,10.1016/j.camwa.2006.10.029}, 
exist for extracting the four Cartesian components of a
quaternion\footnote{These formulae, and that for the quaternion conjugate,
generalize to biquaternions, even with mutually orthogonal unit pure
biquaternions in place of \i, \j and \k, although this latter point has
not been thoroughly checked.}
\cite{Sudbery:1979}:
\begin{equation}
\label{eqn:sudbery}
\begin{alignedat}{2}
w &= \frac{1}{4}  &&\left(q - \i q\i - \j q\j - \k q\k\right)\\
x &= \frac{1}{4\i}&&\left(q - \i q\i + \j q\j + \k q\k\right)\\
y &= \frac{1}{4\j}&&\left(q + \i q\i - \j q\j + \k q\k\right)\\
z &= \frac{1}{4\k}&&\left(q + \i q\i + \j q\j - \k q\k\right)
\end{alignedat}
\end{equation}

The norm of a quaternion is given by the sum of the squares of its components:
$\norm{q} = w^2 + x^2 + y^2 + z^2,\;\norm{q}\in\R$.
It can also be obtained by multiplying the quaternion
by its conjugate, in either order since a quaternion and its conjugate commute:
$\norm{q}=q\qconjugate{q}=\qconjugate{q}q$.
The modulus of a quaternion is the square root of its norm: $\modulus{q} = \sqrt{\norm{q}}$.

Every non-zero quaternion has a multiplicative inverse\footnote{This does not apply to biquaternions.}
given by its conjugate divided by its norm: $q^{-1} = \qconjugate{q}/\norm{q}$.

The quaternion algebra \H is a normed division algebra, meaning that for any two quaternions $p$ and
$q$, $\norm{p\,q} = \norm{p}\,\norm{q}$, and the norm of every non-zero quaternion is non-zero (and
positive) and therefore the multiplicative inverse exists for any non-zero quaternion.

Of course, as is well known, multiplication of quaternions is not commutative, so that in general
for any two quaternions $p$ and $q$, $pq \ne qp$.
This can have subtle ramifications, for example: $(p\,q)^2 = p\,q\,p\,q \ne p^2q^2$.

Alternative representations for quaternions are given in Table \ref{quatrep},
expressed in terms of the Cartesian form given above in equation \ref{classiccartesian},
and in selected cases, in terms of other representations given here.
\begin{table}
\begin{center}
\begin{tabularx}{\textwidth}{X|l|l}
\hline
Designation     & Representation                 & Details ($w, x, y, z, a, b, r, \theta\in\R, q\in\H$,\\
                &                                &          $\mu\in\H, \Scalar{\mu}=0, \mu^2=-1$)\\
\hline\hline
Cartesian       & $q = w + x\i + y\j+ z\k$       & \\ 
\hline
Scalar + vector & $q = \Scalar{q} + \Vector{q}$  & $\Scalar{q} = w$\\
                &                                & $\Vector{q} = x\i + y\j + z\k$\\
\hline  
`Complex' form  & $q = a +\mu b$                 & $a = w$\\
                &                                & $b = \sqrt{x^2 + y^2 + z^2}=\modulus{\Vector{q}}$
                                                   \rule{0pt}{2.5ex}\\ 
                &                                & $\mu = \Vector{q}/\modulus{\Vector{q}}$\\
\hline   
Cayley-Dickson  & $q = (w + x\i) + (y + z\i)\j$\rule{0pt}{2.5ex} 
                                                 & This multiplies out to the Cartesian\\
                &                                & representation.\\
\hline   
Polar form      & $q = r\exp\left(\mu\theta\right) = r\left(\cos\theta+\mu\sin\theta\right)$\rule{0pt}{2.5ex}
                                                 & $r = \modulus{q}$\\
                &                                & $r\cos\theta=a$\\
                &                                & $r\sin\theta=b$\\
\hline
Cayley-Dickson  & $q = \mathcal{A}\exp(\mathcal{B}\j)$              & $\mathcal{A}, \mathcal{B}\in\H\quad\text{with}\quad y=z=0$\\
polar form      &                                & (isomorphic to \C)\\
                &                                & See \cite{10.1007/s00006-008-0128-1,arXiv:0802.0852}
                                                   for formulae defining $\mathcal{A}$ and $\mathcal{B}$.\\
\hline
\end{tabularx}
\caption{\label{quatrep} Representations for quaternions.}
\end{center}
\end{table}
$\mu$ is a unit pure quaternion and is known as the \emph{axis} of the quaternion. It expresses the
direction in 3-space of the vector part\footnote{We use the term `vector part' throughout this paper
to mean that part of the quaternion consisting of the three components containing \i, \j and \k. It
does not necessarily correspond to the concept of a 3-space vector, since it could be a vector,
bivector or a combination of both, using the language of geometric algebra described later. The
term `vector part' is well-established in the literature and, lacking a good alternative
which would be readily understood, we retain it.},
\Vector{q}. Hamilton himself showed that any unit pure
quaternion is a square root of $-1$
\cite[pp\,203, 209]{Hamiltonpapers:V3:7}\nocite{Hamilton:1848}\cite[§\,167, p179]{Hamilton:1853}
and a proof is also given in \cite[Lemma~1]{arXiv:math.RA/0506034}.
This is why we call the form $a + \mu b$ the `complex' form, since it is isomorphic to a complex number
$a+\I b$. This means that the modulus and argument of the quaternion are identical to
those of $a+\I b$ (one could think of them as having similar Argand diagrams, where in
the case of quaternions, the Argand diagram represents a plane section of 4-space defined by
the `axis' $\mu$ and the scalar quaternion axis -- along which $w$ is measured).

The polar form of a quaternion is analogous to the polar form of a complex number, with one
exception. The argument, $\theta$, is confined to the interval $[0,\pi)$ because the modulus
of the vector part is always taken to be positive (there is no convenient way to define an orientation
in 3-space which would permit the sign of the vector part to be determined). If we negate the
argument of the exponential in the polar form, therefore, the negation is conventionally applied
to the axis, $\mu$ and not to the argument $\theta$. When $\theta$ is computed numerically, the
result is always in $[0,\pi)$
because we have to use the (non-negative) modulus of the vector part to compute it
(using an \texttt{atan2} function, typically).

The Cayley-Dickson polar form \cite{10.1007/s00006-008-0128-1,arXiv:0802.0852}
has a complex modulus and a complex argument (both in fact are degenerate quaternions
of the form $w+\i x$, isomorphic to complex numbers).

\section{Biquaternions}
\label{sec:complexq}
To generalize the quaternions to biquaternions we simply permit the four elements to be
complex rather than real, thus giving us the Cartesian representation:
\begin{equation}\label{complexcartesian}
q = W + X\i + Y\j + Z\k
\end{equation}
where $\i$, $\j$ and $\k$ are exactly as in §\,\ref{sec:realq} and $W$, $X$, $Y$, $Z$, are complex.
This generalization was first studied by Hamilton himself 
\cite{Hamiltonpapers:V3:35} and \cite[§\,669, p664]{Hamilton:1853}, and was also
discussed by Cayley \cite{Cayley:1890}.

In equation \ref{complexcartesian} each of the four elements is of the form
$W = \Re(W) + \I\Im(W)$, where $\I^2=-1$
is the usual complex operator, distinct from $\i$. Axiomatically, $\I$ commutes with the three
quaternion operators $\i$, $\j$ and $\k$, that is $\i\I=\I\i$, $\j\I=\I\j$ and $\k\I=\I\k$.
Since reals commute with the three quaternion operators, so do all complex numbers. It is
also worth noting that any quaternion with only one non-zero component in the vector part,
for example $w +\i x$ is not a complex number, but a degenerate quaternion. Such numbers
are isomorphic to the complex numbers, but it is best not to confuse them with complex numbers when
working with biquaternions.

Some familiar rules of algebra apply to biquaternions.
Since the real and imaginary parts are quite separate from the
concept of the four quaternion components,
real and imaginary parts may be equated just as when working with complex equations.
However,
some of the elementary properties of quaternions become
non-elementary when the quaternions are complexified.
For example, generalizations of the norm and modulus of a quaternion
are complex in general for biquaternions,
and so is the argument in one of the polar forms.
We discuss each of these non-trivial properties in a later section.

The existence of complex generalizations of the norm, modulus and inner product,
yields a problem of terminology, since conventionally these quantities are real,
and satisfy properties that their complex generalizations cannot
(the \emph{triangle inequality} $\norm{p+q}\le\norm{p}+\norm{q}$,
for example, requires ordering,
but complex numbers lack ordering,
hence the triangle inequality cannot be applicable to a `norm' with a complex value).
Rather than invent new terms, we use the existing accepted terms (with the
exception of \emph{norm}, where we substitute \emph{semi-norm}, for reasons
discussed in §\,\ref{sec:seminorm}),
but caution the reader that because these quantities
are complex, they cannot be assumed to satisfy all the usual properties of
their conventional real equivalents.
In taking this approach, we are following Synge \cite[p\,9]{Synge:1972},
and to some extent Ward \cite{Ward:1997},
both of whom use the term \emph{norm} to refer to a complex generalization.
Synge also refers to a scalar product with a complex value, and Ward uses the
term \emph{inner product} for a complex generalization of the concept
in his section on the Minkowski metric \cite[§\,3.3, p\,115]{Ward:1997}.

Although a biquaternion commutes with its quaternion conjugate,
and complex numbers commute
(including with their complex conjugates),
somewhat surprisingly,
a biquaternion does not necessarily commute with its complex conjugate
(Proposition \ref{cconjncommute} in §\,\ref{sec:conjugates}).

Alternative representations for biquaternions are shown in Table~\ref{biquatrep},
expressed in terms of the Cartesian form given above in equation \ref{complexcartesian},
and in selected cases, in terms of other representations given here.
\begin{table}[ht]
\begin{center}
\begin{tabularx}{\textwidth}{X|l|l}
\hline
Designation       & Representation                 & Details: $w_a, x_a, y_a, z_a\in\R, a\in\{r,i\}$,\\
                  &                                &          $W, X, Y, Z, A, B, R, \Theta\in\C$,\\
                  &                                &          $q_r, q_i, Q, \Psi\in\H; q\in\B$,\\
                  &                                &          $\xi\in\B, \xi^2=-1$
                                                     \cite{10.1007/s00006-006-0005-8, arXiv:math.RA/0506190}\\
\hline\hline
Cartesian         & $q = W + X\i + Y\j + Z\k$      & \\
\hline
Scalar + vector   & $q = \Scalar{q} + \Vector{q}$  & $\begin{aligned}
                                                      &\Scalar{q} = W\\
                                                      &\Vector{q} = X\i+ Y\j + Z\k\\
                                                      \end{aligned}$\\
\hline
`Complex' form I  & $q = A + \xi B$                & $\begin{aligned}
                                                      &A = W = \Scalar{q}\\
                                                      &\xi B = \Vector{q}\\
                                                      &B = \sqrt{X^2 + Y^2 + Z^2}=\modulus{\Vector{q}}\\
                                                      &\xi = (X\i+ Y\j + Z\k)/B\end{aligned}$\\
\hline
`Complex' form II & $q = q_r + \I q_i
                       = \Re(q) +\I\Im(q)$         & $\begin{aligned}
                                                      q_r &= w_r + x_r\i + y_r\j + z_r\k\\
                                                      q_i &= w_i + x_i\i + y_i\j + z_i\k\\
                                                      \end{aligned}$\\
\hline
Expanded form     & $\begin{aligned}
                     q = \quad&w_r + x_r\i + y_r\j + z_r\k\quad+\\
                          \I(&w_i + x_i\i + y_i\j + z_i\k)
                     \end{aligned}$                & $\begin{aligned}
                                                      w_r &= \Re(W), w_i = \Im(W)\\
                                                      x_r &= \Re(X), x_i = \Im(X)\\
                                                      y_r &= \Re(Y), y_i = \Im(Y)\\
                                                      z_r &= \Re(Z), z_i = \Im(Z)\\
                                                      \end{aligned}$\\
\hline
Hamilton polar form & $\begin{aligned}
                       q &= R\exp\left(\xi\Theta\right)\\
                         &= R\left(\cos\Theta+\xi\sin\Theta\right)\\
                       \end{aligned}$
                                                   & $\begin{aligned}
                                                      R &= \modulus{q}\rule{0pt}{2.5ex}\\
                                                      A &= R\cos\Theta\\                                                                               
                                                      B &= R\sin\Theta\\
                                                      \end{aligned}$\\
\hline
Complex polar form & $\begin{aligned}
                      q &= Q\exp\left(\I\Psi\right)\\
                        &= Q\left(\cos\Psi+\I\sin\Psi\right)\\
                      \end{aligned}$                 & $\begin{aligned}
                                                        \Psi &= \tan^{-1}(q_r^{-1}q_i)\rule{0pt}{2.5ex}\\
                                                        Q &= q/\exp\left(\I\Psi\right)=q\exp\left(-\I\Psi\right)\\
                                                        q_r &= Q\cos\Psi\\
                                                        q_i &= Q\sin\Psi\\
                                                        \end{aligned}$\\
\hline
\end{tabularx}
\caption{\label{biquatrep} Representations for biquaternions.}
\end{center}
\end{table}
In the scalar/vector form, the scalar part is complex,
and the vector part is a pure biquaternion.

Complex form I corresponds to the `complex' form in Table \ref{quatrep}.
The differences are that $A$ and $B$ are now complex,
whereas $a$ and $b$ were real,
and the imaginary unit $\xi$ is a pure biquaternion root of $-1$ of the form
$\xi = b\mu+d\I\nu$, where $b^2-d^2=1$ and $\mu$ and $\nu$ are mutually
perpendicular unit pure quaternions (themselves also roots of $-1$)
\cite{arXiv:math.RA/0506190,10.1007/s00006-006-0005-8}.
Note that $B$ can vanish
-- as discussed in §\,\ref{sec:nilpotents} and in detail in
\cite[§\,2]{arXiv:0812:1102,10.1007/s00006-xxxx-xxxx-x}
-- although its components ($X$, $Y$ and $Z$),
and therefore $q$, do not vanish.
$\xi$ can be thought of as a complex axis,
in the sense that it has real and imaginary parts which each define directions in 3-space.
The geometric interpretation of biquaternions is discussed further in §\,\ref{sec:geometric}.

Complex form II has quaternions in the real and imaginary parts,
and is perhaps the most obvious representation for a biquaternion other
than the Cartesian form.
It is related to the complex polar form described below.

In the expanded form,
the biquaternion is represented as a complex number with quaternion real and imaginary parts
expressed in Cartesian form.

The polar form of a quaternion depends on Euler's formula $\exp(\I\theta)=\cos\theta+\I\sin\theta$
which generalizes by replacing the complex operator \I with any root of $-1$.
In the case of quaternions,
the set of unit pure quaternions provides an infinite number of roots of $-1$
and the general polar form $r\exp(\mu\theta)$, as given in Table \ref{quatrep}
is therefore a straightforward extension of Euler's formula.
In the case of biquaternions there are two possibilities for the root of $-1$:
the complex root $\I$,
or any one of the biquaternion roots of $-1$ defined in §\,\ref{sec:roots}.
Thus we have two possible fundamental polar forms.

The first (`Hamilton') polar form generalizes the polar form in the real case.
$R$, the `modulus' of the biquaternion, is complex; $\xi$ is a biquaternion root of $-1$;
and $\Theta$, the `angle' in the exponential, is also complex\footnote{
The cosine and sine of a complex angle are simply defined in terms of Euler's formula as
$\cos z = \frac{1}{2}\left(e^{\I z}+e^{-\I z}\right)$ and
$\sin z = -\frac{1}{2}\I\left(e^{\I z}-e^{-\I z}\right)$; or, if we write $z = x + \I y$:
$\cos z = \cos x\cosh y - \I\sin x\sinh y$ and $\sin z = \sin x\cosh y + \I\cos x\sinh y$
\cite[See \textbf{complexes}]{Bouvier:2e}.}.
The interpretation of complex angles is discussed further in §\,\ref{sec:inner}.

In the second (`complex') polar form,
the standard complex operator \I serves as the root of $-1$ in the exponential,
but in this form, the `argument' of the exponential, $\Psi$, is a quaternion,
and the exponential is scaled by a quaternion `modulus' $Q$.
We thus have a polar form with a `modulus' and `argument' in \H.
Note that it is not possible to find $Q$ by the obvious direct route of
$Q = \sqrt{q_r^2+q_i^2}$ because $q_r = Q\cos\Psi$
and squaring this gives $Q\cos\Psi\,\,Q\cos\Psi$
and not $Q^2\cos^2\Psi$.
However, it is the case that $\cos^2\Psi+\sin^2\Psi=1$, as would be expected.
Further, note that care is needed in computing the inverse tangent in order to find $\Psi$:
it is important that the real part is divided on the left.
This is because
\begin{align*}
\tan\Psi &= (Q\cos\Psi)^{-1}(Q\sin\Psi) = \cos^{-1}\!\Psi\,\,Q^{-1}Q\,\sin\Psi = \cos^{-1}\!\Psi\,\sin\Psi
\intertext{whereas with a division on the right we have:}
\tan\Psi &\ne (Q\sin\Psi)(Q\cos\Psi)^{-1} = Q\sin\Psi\cos^{-1}\!\Psi\,\,Q^{-1}
\end{align*}
and $Q$ and its inverse are at opposite ends of the product and do not cancel.
(Note that $\cos\Psi$ and $\sin\Psi$ commute, so their quotient is the same
whether divided on the left or right.)
The exponential is a biquaternion because of the presence of \I and it has a complex modulus.
The modulus of this complex modulus is 1.
We discuss both polar forms further in §\,\ref{sec:conjugates} in the context of conjugation.
Finally, note that in the complex polar form, $Q$ and the exponential do not commute.
It is therefore possible to define a variant by placing $Q$ on the right of the exponential.
The variant is related to the form in Table~\ref{biquatrep} by the conjugate rule.

De~Leo and  Rodrigues \cite{arxiv:hep-th/9806058,10.1023/A:1026692508708}
discussed polar forms of biquaternions but apparently did not see that there
were two simple polar forms as here, with different imaginary units.
Instead they described a single polar form containing the product of two exponentials.
We can do this with either of our polar forms, representing the `modulus' in each case
in its own polar form. Thus the `Hamilton' polar form can be written:
\begin{equation}
\label{hamiltonpolar}
q = R\exp\left(\xi\Theta\right) = r\exp\left(\I\phi\right)\exp\left(\xi\Theta\right)
\end{equation}
where $r,\phi\in\R$ are the modulus and argument of $R$, the complex `modulus' of $q$.
Note that, because \I and $\xi$ commute, the two exponentials commute,
and it is possible to write this polar form as:
\[
q = r\exp(\I\phi + \xi\Theta)
\]
where the argument of the exponential is now a biquaternion with $\I\phi$ as scalar
part and $\xi\Theta$ as vector part.
(In general, with quaternions as well as biquaternions, $e^p\,e^q\ne e^{pq}$
because of non-commutativity.)

Similarly, the `complex' polar form can also be written in this way, expanding the
quaternion `modulus', $Q$, into the standard polar form of a quaternion as given in
Table~\ref{quatrep}:
\begin{equation}
\label{complexpolar}
q = Q\exp\left(\I\Psi\right) = r\exp\left(\mu\theta\right)\exp\left(\I\Psi\right)
\end{equation}
where $\mu\in\H$ and $r,\theta\in\R$.
Notice that the single real modulus $r$ is the same in each case,
but the various `angles' are different in value and type ($\phi,\theta\in\R, \Theta\in\C, \Psi\in\H$).
The real modulus, $r$, is the absolute value of the square root of the semi-norm
as discussed in §\,\ref{sec:seminorm}.
In this case, the arguments of the two exponentials
(specifically $\mu$ and $\Psi$)
do not commute,
hence the two exponentials cannot be combined by adding $\mu\theta$ to $\I\Psi$,
and neither can the order of the exponentials be changed.

We can usefully combine the `complex' representation of a quaternion from Table \ref{quatrep} with
`complex' form II from Table \ref{biquatrep} to give the following
representation, which was used in \cite{10.1007/s00006-006-0005-8, arXiv:math.RA/0506190} to derive the
biquaternion roots of $-1$:
\begin{equation}
\label{dcomplex}
q = (\alpha + \mu\beta) + \I(\gamma + \nu\delta)
\end{equation}
in which $\mu$ and $\nu$ are real pure unit quaternions, and $\alpha$, $\beta$, $\gamma$ and
$\delta$ are real. The four real coefficients in this representation may be related to the
coefficients in the other representations given above. For example: $\alpha + \I\gamma = A = W$.
The correspondence between $\xi$, and $\mu$ and $\nu$ is not so simple. Equating the vector part
of equation \ref{dcomplex} with the vector part of `complex' form I, we get:
\[
\mu\beta + \I\nu\delta = \xi B = \xi\sqrt{X^2 + Y^2 + Z^2}
\]
and we can see that dividing $\mu\beta + \I\nu\delta$ by its (complex) modulus $B$, will yield $\xi$
\emph{provided that $B$ does not vanish, as discussed in §\,\ref{sec:nilpotents}}.

We may relate each of the terms in equation \ref{dcomplex} to a concept from geometric algebra
\cite{HestenesSobczyk:1984,Sommer:2001},
as shown in Table~\ref{correspondence} with equivalent definitions based on various representations
from Table~\ref{biquatrep}.
\begin{table}[ht]
\begin{center}
\begin{tabularx}{\textwidth}{X|c|r@{\,}c@{\,}r@{\,}c@{\,}r@{\,}c@{\,}l}
\hline
Geometric algebra concept & Term in equation \ref{dcomplex} &
\multicolumn{7}{c}{Terms from Table \ref{biquatrep}}\\
\hline
\hline
Scalar       -- undirected quantity & $\alpha$      & $\Re(W)$      &=&$\Re(A)$           &=&$\Scalar{q_r}$ &=& $w_r$\\
Bivector     -- directed area       & $\mu\beta$    & $\Re(\xi B)$  &=&$\Re(\Vector{q})$  &=&$\Vector{q_r}$&&\\
Vector       -- directed magnitude  & $\I\nu\delta$ & $\I\Im(\xi B)$&=&$\I\Im(\Vector{q})$&=&$\I\Vector{q_i}$&&\\
Pseudoscalar -- undirected volume   & $\I\gamma$    & $\I\Im(W)$    &=&$\I\Im(A)$         &=& $\I\Scalar{q_i}$ &=& $\I w_i$\\
\hline
\end{tabularx}
\end{center}
\caption{\label{correspondence}Correspondence between geometric algebra concepts
                               and biquaternion components.}
\end{table}

These equivalences are given by Ward \cite[§\,3.2, p\,112]{Ward:1997}
and they are discussed in more detail in §\,\ref{sec:geometric}.
Note carefully that geometric \emph{vectors} are represented by imaginary pure quaternions,
and that real pure quaternions are \emph{bivectors}.
This is because the product of two perpendicular vectors must yield a bivector.
The product of two bivectors gives a bivector
(and a scalar, unless the bivectors are perpendicular).

The representation in equation \ref{dcomplex} was used in
\cite{10.1007/s00006-006-0005-8,arXiv:math.RA/0506190}
to derive the solutions of the equation $q^2 = -1$ (that is
the biquaternion roots of $-1$) when $q$ is a biquaternion, and it was shown
that the solutions required $\alpha=\gamma=0$, $\mu\perp\nu$, and $\beta^2-\delta^2=1$. Clearly any
of the other representations given in this section would not permit this result to be expressed so
clearly, nor to be easily derived. We return to this result in Theorem \ref{rootsofunity}.

The biquaternion algebra $\B$ is not a division algebra because non-zero biquaternions
exist that lack a multiplicative inverse.
The set of such biquaternions is known as the \emph{divisors of zero}
\cite{arXiv:0812:1102,10.1007/s00006-xxxx-xxxx-x} and is defined in §\,\ref{sec:divisors}.

\subsection{Conjugates}
\label{sec:conjugates}
The conjugate of a biquaternion may be defined exactly as for a quaternion by
negating the vector part. Thus we have
\begin{equation}
\label{qconjugate}
\qconjugate{q} = W - X\i - Y\j - Z\k = \Scalar{q} - \Vector{q} = \qconjugate{q_r} + \I\qconjugate{q_i}
\end{equation}
We call this the `Hamiltonian' or \emph{quaternion conjugate}, in agreement with Synge
\cite[Equation 3.4, p\,8]{Synge:1972}.
The conjugate rule $\qconjugate{q}\,\qconjugate{p} = \qconjugate{p\,q}$ and its generalization
to more than two quaternions applies equally to biquaternions.
Similarly,
a biquaternion commutes with its quaternion conjugate,
and the product of the two is the semi-norm,
as discussed in §\,\ref{sec:seminorm}.

However, there is another possible conjugate -- that obtained by taking the complex
conjugate of the complex components of the quaternion. We denote this complex conjugate by
a superscript star, which is a common convention with complex numbers. Thus the \emph{complex
conjugate} is given by:
\begin{equation}
\label{cconjugate}
\cconjugate{q} = \cconjugate{W} + \cconjugate{X}\i + \cconjugate{Y}\j + \cconjugate{Z}\k
               = \cconjugate{\Scalar{q}} + \cconjugate{\Vector{q}} 
               = q_r - \I q_i
\end{equation}
Our definition in equation~\ref{cconjugate} agrees with that of Synge \cite[p8, equation
3.4]{Synge:1972} and appears to be the obvious way to define the complex conjugate, but it differs
from that of Ward \cite{Ward:1997} who defines a complex conjugate which is a quaternion conjugate
with complex conjugate components.
There is no way to define the complex conjugate in terms of additions and multiplications
as can be done with the quaternion conjugate --- if a means existed, it would also apply to complex
numbers because a degenerate biquaternion of the form $\alpha+\I\gamma$ with $\alpha,\gamma\in\R$,
is isomorphic to a complex number.
Instead the complex conjugate must be seen as a fundamental operation.
Note very carefully that a biquaternion \emph{does not commute} with its complex conjugate,
a surprising result that we examine in Proposition~\ref{cconjncommute} below.

Of course, we can apply both conjugates \cite{Ward:1997} (we call this a \emph{total} conjugate,
but the term \emph{biconjugate} has also been suggested in \cite{arXiv:physics/0508036},
and we consider it apposite).
It is not difficult to show the following results:
\begin{align*}
\qconjugate{\cconjugate{q}} &= \cconjugate{\qconjugate{q}\,}\\
\cconjugate{(p q)}          &= \cconjugate{p}\cconjugate{q}
\end{align*}
In terms of the various representations above, the total or biconjugate is:
\begin{equation}
\label{biconjugate}
\qconjugate{\cconjugate{q}} = w_r - x_r\i - y_r\j - z_r\k - \I(w_i - x_i\i - y_i\j - z_i\k)
                            = \qconjugate{q_r} - \I\,\qconjugate{q_i}
                            = \cconjugate{\Scalar{q}} - \cconjugate{\Vector{q}}
\end{equation}
However,
the ramifications of non-commutativity are deep,
as shown by the following Proposition.
\begin{proposition}
\label{cconjncommute}
A biquaternion does not, in general, commute with its complex conjugate:
reversing the order of the product yields a complex conjugate result.
\end{proposition}
\begin{proof}
Represent an arbitary biquaternion in complex form I, as given in Table \ref{biquatrep}:
$q = q_r + \I q_i$.
Then the two products of $q$ with its complex conjugate $\cconjugate{q}$ are:
\begin{align*}
q\,\cconjugate{q}\!&= \left(q_r +\I q_i\right)\left(q_r -\I q_i\right)=q_r^2 + q_i^2 + I(q_i q_r - q_r q_i)\\
\cconjugate{q} q   &= \left(q_r -\I q_i\right)\left(q_r +\I q_i\right)=q_r^2 + q_i^2 - I(q_i q_r - q_r q_i)
\end{align*}
The results are a complex conjugate pair, and therefore differ.
\end{proof}
Note that important exceptions to Proposition \ref{cconjncommute} are:
\begin{itemize}
\item Quaternions, that is biquaternions with $q_i=0$ (trivial).
\item Imaginary biquaternions, with $q_r=0$.
      In this case the product is $q_i^2$, regardless of order.
\item Biquaternions with real and imaginary parts that commute
      (co-planar, with a common axis).
      In this case the imaginary part of the product vanishes and the
      result is $q_r^2 + q_i^2$, regardless of order.
\end{itemize}

The Hamilton and complex conjugates correspond to negation of the argument of the
exponential in the two polar forms in Table~\ref{biquatrep}.
In the `Hamilton' polar form, negating the argument of the exponential negates the sine term,
which corresponds to the vector part of $q$, and thus yields the classical Hamilton conjugate.
In the `complex' polar form, negating the argument of the exponential again negates the sine term,
but in this case the sine term corresponds to the imaginary part of the quaternion,
and thus it is the complex conjugate that is obtained.
A third possibility would be a polar form that would correspond
(in the sense just outlined)
to the total conjugate.
However,
this would require the sine term to represent
all the components in equation \ref{biconjugate} except $w_r$,
which would be represented in the cosine term.
This appears improbable.

The three types of conjugate allow us to construct formulae for extracting the geometric
components of a biquaternion, as defined in Table \ref{correspondence}.
The formulae in Table \ref{geomconj} are obtained by substitution from four formulae for
the scalar/vector and real/imaginary parts of a quaternion based on sums and differences
of the Hamilton or complex conjugates.
Unlike the formulae in equation~\ref{eqn:sudbery} however,
which are based on involutions,
and therefore require multiplication and addition only,
these formulae contain complex conjugates,
which are primitive operations that cannot be expressed in terms of multiplications and additions.
Similarly, and fairly trivially,
by combining the formulae in equation~\ref{eqn:sudbery} with formulae based on
sums and differences of complex conjugates,
one can obtain formulae for extracting any of
the eight real components of a biquaternion\footnote{The utility of
such formulae is not in numeric computation,
where explicit access to the components is a much better approach,
but in algebraic manipulation, where the formulae may allow an
algebraic solution to an otherwise seemingly difficult derivation.}.
For example, that for the real part of the scalar part is
$w_r=\frac{1}{8}(q + \cconjugate{q}-\i q\i-\i\cconjugate{q}\i
                                   -\j q\j-\j\cconjugate{q}\j
                                   -\k q\k-\k\cconjugate{q}\k)$.
\begin{table}[htb!]
\begin{center}
\begin{tabular}{l@{\quad}l|@{\rule[-1.5ex]{0pt}{4ex}\hspace{1ex}}c}
\hline
Scalar & $\Re(\Scalar{q})$ &
$\frac{1}{4}\left(q + \qconjugate{q} + \cconjugate{q} + \cconjugate{\qconjugate{q}}\right)$\\
\hline
Bivector & $\Re(\Vector{q})$ &
$\frac{1}{4}\left(q - \qconjugate{q} + \cconjugate{q} - \cconjugate{\qconjugate{q}}\right)$\\
\hline
Vector & $\Im(\Vector{q})\I$ &
$\frac{1}{4}\left(q - \qconjugate{q} - \cconjugate{q} + \cconjugate{\qconjugate{q}}\right)$\\
\hline
Pseudoscalar & $\Im(\Scalar{q})\I$ &
$\frac{1}{4}\left(q + \qconjugate{q} - \cconjugate{q} - \cconjugate{\qconjugate{q}}\right)$\\
\hline
\end{tabular}
\caption{\label{geomconj}Formulae for the geometric components of a biquaternion.
                         (See also Table \ref{correspondence}.)}
\end{center}
\end{table}

\subsection{Inner Product}\label{sec:inner}
The definition of the inner product in the classic quaternion case was given by Porteous
\cite[Prop.\,10.8, p\,177]{Porteous:1981} and we take this as the basis for defining the
inner product of two biquaternions, since it is consistent with the quaternion case, and
also yields the semi-norm (see the next section) in the case of the inner product of a
biquaternion with itself.
The same formula is given by Synge \cite[Equation~3.8, p\,9]{Synge:1972} in the context of
a discussion of basic properties of biquaternions.
Ward, in his discussion of biquaternions and the
Minkowski metric \cite[§\,3.3, p114]{Ward:1997}, utilises the same definition, but he also
discusses other definitions of the inner product (see \cite[§\,3.2, p109]{Ward:1997}).
The definitions of Porteous and Synge are:
\begin{equation}\label{innerprod}
\inner{\,p}{q} = \frac{1}{2}\left(\,\qconjugate{p}\,q + \qconjugate{q}\,p\right)
               = \frac{1}{2}\left(p\,\qconjugate{q} + q\,\qconjugate{p}\right)
\end{equation}
where the overbar represents a (Hamilton or quaternion) conjugate.
The result of this expression will have zero vector part
(since the vector parts of the two terms inside the parentheses cancel).
In general, the scalar part of the result, and therefore the inner product, will be complex.
The inner product may also be
defined in terms of a simple elementwise product of the two quaternions, as can be
seen by expanding out equation \ref{innerprod}. If we let
$p = W_p + X_p\i + Y_p\j + Z_p\k$ and $q = W_q + X_q\i + Y_q\j + Z_q\k$, then
\[
\inner{\,p}{q} = W_p W_q + X_p X_q + Y_p Y_q + Z_p Z_q
\]
which is, of course, complex. (The result will be equal to the scalar part of the quaternion
with zero vector part resulting from equation \ref{innerprod}.)

As already discussed in §\,\ref{sec:complexq}, the use of the term \emph{inner product}
here does not imply that all the usual properties of an inner product will be satisfied.
Ward \cite[§\,3.3, p\,115]{Ward:1997} states the following properties that are satisfied
by the inner product defined in equation \ref{innerprod}, and these are easily verified:
\begin{align*}
\inner{p}{q}       &= \inner{q}{p},& p,q\in\B\\
\inner{p}{q+r}     &= \inner{p}{q}+\inner{p}{r},& r\in\B\\
\alpha\inner{p}{q} &= \inner{\alpha p}{q} = \inner{p}{\alpha q},&\alpha\in\C
\end{align*}
Conventionally, but not always \cite[See: \textbf{Scalar product}]{PenguinDictMaths3e},
the inner product is \emph{positive definite}, that is greater than or equal to zero,
or non-negative, which cannot apply to a complex-valued inner product.

Classically, the inner product of two vectors is given by $\modulus{v_1}\modulus{v_2}\cos\theta$
where $\theta$ is the angle between the two vectors.
If extended to quaternions,
it is not difficult to see that the angle is that between the two quaternions in 4-space,
in the common plane defined by the two quaternions.
When we extend this concept to biquaternions, the angle between them, and
their moduli, must be complex, in general, so the geometric interpretation
of the inner product is slightly more difficult.
In the quaternion case, orthogonality arises from the angle between the two
quaternions (having a vanishing cosine), but in the biquaternion case there
is the additional possibility that one or both moduli are zero, resulting in
a vanishing inner product.
Expressing the inner product of two biquaternions in terms of inner products
of the real and imaginary parts makes it possible to understand what the inner
product represents geometrically.
Representing $p$ as follows: $p = p_r + \I p_i$ and similarly for $q$,
the inner product may be expanded as:
\[
\inner{\,p}{q} = \inner{\,p_r}{q_r} - \inner{\,p_i}{q_i}
    +\I\left(\,\inner{\,p_r}{q_i} + \inner{\,p_i}{q_r}\,\right)
\]
Certain special cases are apparent after careful inspection:
\begin{itemize}
\item $\inner{p}{q} = 0$ indicates that the two biquaternions are orthogonal or
      `perpendicular', and it requires the real and imaginary parts of the inner
      product to vanish separately.
      The detailed conditions required for this are many, since the
      four inner product components (two real, two imaginary) may have positive or
      negative real values and can cancel in several ways.
      We list here four different ways in which the scalar product can vanish:
      \begin{description}
      \item[The strongest constraint] is when all four real and imaginary parts of
      the two biquaternions are mutually orthogonal.
      A simple example is $p = 1 + \I\i$ and $q = \j + \I\k$, but it is easy
      to construct a more general example starting from a random quaternion.
      Let $p_1$ be a randomly chosen (pure or full) quaternion.
      Then let $p_2 = p_1\i$, $p_3 = p_1\j$ and $p_4 = p_1\k$.
      The four quaternions thus constructed have the same four numeric components
      permuted with sign changes in such a way that any two will have a vanishing
      inner product\footnote{Three other permutations can be obtained by multiplication
      on the right, or by division on either side by $\i$, $\j$ and $\k$.}.
      Two biquaternions $p = p_1 + \I p_2$ and $q = p_3 +\I p_4$
      (or any other permutation)
      will have $\inner{p}{q} = 0$.
      This pair of biquaternions will be found to be divisors of zero (see §\,\ref{sec:divisors}),
      but scaling the real and imaginary parts with different scale factors will
      result in biquaternions that are not divisors of zero, but are still
      orthogonal (orthogonality not being dependent on norm).
      Choosing the initial random value to be pure does not result in the pair
      of biquaternions being pure, because of permutation of the components.
      \item[Weaker constraint I]\mbox{} 
      \begin{enumerate}
      \item The real parts of the two biquaternions are orthogonal, and
            the imaginary parts of the two biquaternions are orthogonal, hence
            the real part of the scalar product vanishes because it is the difference
            of two vanishing scalar products.
      \item The real part of each biquaternion is not orthogonal to the imaginary part
            of the other. However, the scalar products of the real part of one
            biquaternion with the imaginary part of the other have the same values, but
            \emph{opposite signs}.
            This means the imaginary part of the scalar product
            of the two biquaternions vanishes \emph{by cancellation}.
      \end{enumerate}
      A pair of biquaternions satisfying this constraint can be constructed from
      an arbitrary biquaternion $p$, by choosing $q = p\,\i$ (and two others also
      orthogonal to the first may be constructed by choosing $q = p\j$ or
      $q = p\k$)\footnote{As in
      the stronger case, it is also possible to multiply by \i on the left, \textit{etc}.}.
      \item[Weaker constraint II]\mbox{} 
      \begin{enumerate}
      \item The real parts of the two biquaternions are not orthogonal,
      and neither are the imaginary parts, but the scalar product of the two real
      parts has the same value \emph{and sign} as the scalar product of the two imaginary parts.
      This means the real part of the scalar product of the two biquaternions
      vanishes \emph{by cancellation}.
      \item The real part of each biquaternion is orthogonal to the imaginary part
      of the other, hence the imaginary part of the
      scalar product of the two biquaternions vanishes because it is the sum of
      two vanishing scalar products.
      \end{enumerate}
      A pair of biquaternions satisfying this constraint can be constructed from
      an arbitrary biquaternion $p$, by choosing $q = p\,\I\i$ (and ditto,
      \textit{mutatis mutandis.}).
      \item[The weakest constraint] is when none of the real and imaginary parts
      of the two biquaternions are orthogonal, but the real and imaginary parts of
      the scalar product of the two biquaternions vanish by cancellation.
      A pair of biquaternions satisfying this constraint can be constructed
      from an arbitrary biquaternion $p$, by choosing $q = p (\i + \I\j)$ or similar.
      Note that this results in $q$ being a divisor of zero (see §\,\ref{sec:divisors})
      because $(\i + \I\j)$ is a divisor of zero.
      \end{description}
\item $\inner{p}{q}$ is real, that is with zero imaginary part.
      There is a trivial case: the two biquaternions may be imaginary with
      zero real parts. Otherwise, 
      the imaginary part of the inner product can vanish in two ways,
      as in the weaker cases above of $\inner{p}{q}=0$:
      either the imaginary part vanishes by cancellation,
      or it vanishes because the real part of each biquaternion is orthogonal to
      the imaginary part of the other.
      It is easy to construct a biquaternion to satisfy the second condition
      using the technique outlined above for the strongest constraint: construct
      two pairs of orthogonal quaternions, and then construct two biquaternions
      using the components of these pairs.
\item $\inner{p}{q}$ is imaginary, that is with zero real part.
      The conditions required for this are very similar to the previous case,
      with the appropriate changes.
      A trivial case is where one biquaternion is real and the other is imaginary.
\end{itemize}
Clearly, from the above analysis, \emph{orthogonality} of biquaternions is not
as simple as orthogonality of real quaternions where a geometrical interpretation
in terms of a plane and a real angle is straightforward.
In the biquaternion case, interpreting the inner product $\inner{p}{q}$ as
$\modulus{p}\modulus{q}\cos\Theta$, where the two moduli and the angle are
complex, is not at all obvious.
This remains a topic for further work.

\subsection{Semi-norm}\label{sec:seminorm}
It is possible to define more than one norm for the biquaternions
(see for example \cite[§§\,3.2 and 3.3]{Ward:1997}).
In this section we consider the generalization of the quaternion norm
$\norm{q} = w^2 + x^2 + y^2 + z^2,\;\norm{q}\in\R$ to the biquaternion
case by allowing the four Cartesian components to become complex.
Conventionally a \emph{norm} is real and positive definite, that is non-negative,
but in the case of biquaternions, this convention must be relaxed because the
norm has a complex value.
Additionally, the existence of divisors of zero (see §\,\ref{sec:divisors}) means that a non-zero
biquaternion can have a vanishing norm.
For this reason, we use the term \emph{semi-norm}.
A semi-norm is a generalization of the concept of a norm, with no requirement that the norm be
zero only at the origin
(of a vector space) \cite[See: \textbf{semi-norm}]{CollinsDictMaths}.
As with our use of the term \emph{inner product} in §\,\ref{sec:inner}
for a complex-valued generalization of the inner product,
we are here extending the term \emph{semi-norm} to a complex-valued quantity analogous
to a norm, with the additional property of vanishing in the case of divisors of zero.
We also use the term \emph{modulus} for the square root of the semi-norm -- it is also
of course complex, in general.

The semi-norm can be real and negative in special cases, as well as purely imaginary.
The semi-norm can be defined in terms of the inner product of a biquaternion with itself,
that is \norm{q} = \inner{q}{q}
or directly in terms of the four complex components:
$\norm{q} = \modulus{q}^2 = W^2 + X^2 + Y^2 + Z^2$.
Although the semi-norm is complex-valued
the result given by Coxeter for quaternions \cite[§\,2]{Coxeter:1946} still holds:
$\norm{q} = \modulus{q}^2 = q\,\qconjugate{q} = \qconjugate{q}\,q$.
This is, of course, a special case of the formula for the inner product,
in this case of $q$ with itself (equation \ref{innerprod}).

\begin{lemma}\textnormal{\cite[Lemma 1]{arXiv:0812:1102,10.1007/s00006-xxxx-xxxx-x}}\label{normlemma}
Let $q = q_r + \I q_i$ be a non-zero biquaternion with real part $q_r\in\H$ and imaginary part $q_i\in\H$.
The real part of $\norm{q}$ is equal to the difference between the norms of $q_r$ and $q_i$,
and the imaginary part of $\norm{q}$ is equal to twice the inner product of $q_r$ and $q_i$.
\end{lemma}
\begin{proof}
We express the semi-norm of $q$ as the product of $q$ with its quaternion conjugate:
$\norm{q} = q\qconjugate{q}$. Writing this explicitly:
\begin{align*}
\norm{q} &= (q_r + \I q_i)\qconjugate{(q_r + \I q_i)} = (q_r + \I q_i)(\qconjugate{q_r} + \I\qconjugate{q_i})\\
\intertext{and multiplying out we get:}
\norm{q} &= q_r\qconjugate{q_r} - q_i\qconjugate{q_i} + \I(q_r\qconjugate{q_i} + q_i\qconjugate{q_r})
\end{align*}
The real part of \norm{q} can be recognised as $\norm{q_r} - \norm{q_i}$
from Coxeter's result \cite[§\,2]{Coxeter:1946}.
The imaginary part of \norm{q} is twice the inner product of $q_r$ and $q_i$ as given by
Porteous \cite[Prop.\,10.8, p\,177]{Porteous:1981} and equation \ref{innerprod}.
\end{proof}
There are some special cases of the result given in Lemma \ref{normlemma}:
\begin{itemize}
\item biquaternions with perpendicular real and imaginary parts have a real semi-norm
      (because the imaginary part vanishes), but the semi-norm may be negative
      if the norm of the imaginary part exceeds the norm of the real part;
\item biquaternions with real and imaginary parts with equal norms have an imaginary semi-norm
      (because the real part vanishes).
\item imaginary biquaternions (with $q_r = 0$) have a negative real semi-norm,
      and therefore an imaginary modulus.
\item biquaternions with perpendicular real and imaginary parts with equal norms
      have a vanishing semi-norm. This important special case is covered in more
      detail in §\,\ref{sec:divisors}.
\end{itemize}

The semi-norm is invariant under quaternion conjugation, that is $\norm{q}=\norm{\qconjugate{q}}$,
but not under complex conjugation: $\norm{p}=\cconjugate{\norm{\cconjugate{p}}}$.
This can be seen easily from Lemma \ref{normlemma}, since quaternion conjugation
does not affect the norms of the real and imaginary parts, nor their inner product.
However, although complex conjugation does not alter the norms of the real and imaginary parts,
it does negate their inner product and therefore negates the imaginary part of the semi-norm.

The semi-norm as defined here
(and by Ward \cite[§\,3.3, p\,115]{Ward:1997})
obeys the `rule of the norms'\footnote{The `rule of the norms' holds also for the modulus
(the square root of the norm), which is what the term `norm' refers to in many sources in
the literature.} so that for any two biquaternions,
the semi-norm of their product equals the product of their semi-norms:
$\norm{p\,q}=\norm{p}\norm{q}$.
This is true even in the case where one or both of $p$ and $q$ is a divisor of zero,
with vanishing semi-norm,
as discussed in §\,\ref{sec:divisors}
(but the result will be zero, computed as the semi-norm of the product,
or the product of the semi-norms).
Hence $\B$ is a \emph{normed algebra} using the definition of the semi-norm
used in this section.
The non-zero biquaternions that have zero semi-norm
(the \emph{divisors of zero})
lack a multiplicative inverse.

A \emph{norm} is usually required to satisfy the following three properties
\cite[See: \textbf{norm}]{CollinsDictMaths} (caution\footnote{The properties given in
equation \ref{eqn:normprops} are stated using the notation used in the
cited reference which does not conform to the usage elsewhere in this paper.
In particular, $\norm{x}$ in equation \ref{eqn:normprops} means what we call
the modulus or square root of the norm in this paper.}):
\begin{equation}
\label{eqn:normprops}
\begin{aligned}
\norm{-x}         &= \norm{x}\\
\norm{\lambda\,x} &= \modulus{\lambda} \norm{q}\\
\norm{x + y}      &\le\norm{x} + \norm{y}\quad\text{(Triangle inequality)}
\end{aligned}
\end{equation}
The first of these properties is easily seen to be satisfied for all $q\in\B$
because negating a complex number does not change its square, and therefore
does not alter the sum of the squares of the four components of a biquaternion,
nor the square root of the sum of the squares.
The second property requires rather more careful analysis, but when the
notation is correctly interpreted, it can be seen to be satisfied by the
square root of the semi-norm of a biquaternion. The quantity $\lambda$ is
usually taken to be a scalar, that is in the biquaternion context,
a complex number.
If we write the expression explicitly in biquaternion terms,
using the Cartesian form, we have:
$\modulus{\lambda q}=\sqrt{(\lambda W)^2 + (\lambda X)^2 +
                           (\lambda Y)^2 + (\lambda Z)^2}
                    =\sqrt{\lambda^2}\modulus{q}$.
Taking the square root of the square of a scalar quantity conventionally
yields the positive square root, or absolute value.
In the case of a complex $\lambda$, taking the square root of the square
yields a result in the right half-plane, rather than the absolute value.
Thus to decide whether the biquaternion semi-norm satisfies the property
as conventionally stated requires some deeper understanding of the reason
for the property. 
The triangle inequality, as has already been noted in §\,\ref{sec:complexq}
is not applicable to the biquaternion semi-norm, because ordering is not
defined for complex numbers.

Finally we note that it is possible to define a real norm and modulus for the biquaternions.
Gürlebeck and Sprössig \cite[Lemma 1.30]{Gurlebeck} state that there is a unique real norm
satisfying $\norm{pq} = \norm{p}\,\norm{q}$ and $\norm{\oldmu\sigma_0} = \modulus{\oldmu}, \oldmu\in\R$
where $\sigma_0$ is a basis element (Pauli matrix).
This norm is simply the absolute value of the square root of the semi-norm,
in other words,
$r$ in equations \ref{hamiltonpolar} and \ref{complexpolar},
or $\modulus{Q}$ in Table~\ref{biquatrep}.

\subsection[Roots of -1]{Roots of $-1$}
\label{sec:roots}
This section draws on and builds on results first presented by Sangwine in
\cite{arXiv:math.RA/0506190,10.1007/s00006-006-0005-8}.
These results have been generalized to Clifford algebras by Hitzer and Ab{\l}amovicz
\cite{Hitzer:TR_2009_3,arXiv:0905.3019} where the concept is aptly referred to as
\emph{geometric roots} of $-1$.

Biquaternion roots of -1 appear as $\xi$ in Table \ref{biquatrep}, in the `complex'
and `Hamilton polar' forms.
\begin{theorem}\label{unitnorm}
Any non-zero quaternion or biquaternion with a non-vanishing modulus,
divided by its modulus has a unit real norm.
\end{theorem}
\begin{proof}
Let $q$ be an arbitrary quaternion or biquaternion and let $p = q/\modulus{q}$.
For any quaternion or biquaternion $x$ we have $\norm{x}=x\qconjugate{x}$ and $\modulus{x}=\sqrt{\norm{x}}$.
Therefore $p = q/\sqrt{q\qconjugate{q}}$ and we can write the norm of $p$ as:
\begin{align*}
\norm{p} &= p\,\qconjugate{p} = \frac{q}{\sqrt{q\,\qconjugate{q}}}
                              \qconjugate{\left(\frac{q}{\sqrt{q\,\qconjugate{q}}}\right)}\\
\intertext{Since $\sqrt{q\,\qconjugate{q}}$ is complex, it is unaffected by the quaternion conjugate
           and we can simplify this to:}
         &= \frac{q}{\sqrt{q\,\qconjugate{q}}}\frac{\qconjugate{q}}{\sqrt{q\,\qconjugate{q}}}
          = \frac{q\,\qconjugate{q}}{q\,\qconjugate{q}} = 1
\end{align*}
\end{proof}
\begin{theorem}\label{rootsofunity}
Any pure quaternion or biquaternion with a non-vanishing modulus,
divided by its modulus, is a root of $-1$.
\end{theorem}
\begin{proof}
Let $q$ be an arbitrary quaternion or biquaternion with $\modulus{q}\ne 0$.
Then $\left(q/\modulus{q}\right)^2 = q^2/q\,\qconjugate{q} = q/\qconjugate{q}$.
If we restrict $q$ to have zero scalar part, the conjugate reduces to negation and we have
$\left(q/\modulus{q}\right)^2 = q/-q = -1$.
\end{proof}
Theorem \ref{rootsofunity} is simply a restatement of the well-known fact that unit pure
quaternions are roots of $-1$ \cite[§\,167, p179]{Hamilton:1853},
but in the case of biquaternions,
the ramifications of the result are not so obvious.
Theorem \ref{rootsofunity} is also easily demonstrated using the Cartesian form:
\begin{align}
\left(\frac{X\i+Y\j+Z\k}{\sqrt{X^2+Y^2+Z^2}}\right)^2 &=
\frac{-X^2 + XY(\i\j+\j\i) + XZ(\i\k+\k\i) -Y^2 + YZ(\j\k+\k\j) -Z^2}{X^2+Y^2+Z^2}\nonumber\\
\label{simpleroot}
&= \frac{-\left(X^2+Y^2+Z^2\right)}{X^2+Y^2+Z^2} = -1
\end{align}

It was shown in \cite[Theorem 2.1]{10.1007/s00006-006-0005-8}\footnote{Also in 
                \cite[Theorem 1]{arXiv:math.RA/0506190}.} that any pure biquaternion
$\xi$ satisfying the constraints $\Re(\xi)\perp\Im(\xi)$ and $\norm{\Re(\xi)}-\norm{\Im(\xi)}=1$ is
a root of $-1$ and that no other biquaternions are roots of $-1$.
It follows therefore that dividing an arbitrary pure biquaternion by its
(complex) modulus produces a result that satisfies the constraints stated.
Since these constraints are by no means obvious from equation \ref{simpleroot},
it is interesting to verify them directly.
To verify the constraints stated, we take an arbitrary pure biquaternion $p$,
divide it by its (complex) modulus and show that the result satisfies the constraints stated.
Dividing $p$ by its modulus gives a root of $-1$: $\xi = p/\modulus{p}$.
We have to show that the real and imaginary parts of $\xi$ are orthogonal,
that is $\inner{\Re(\xi)}{\Im(\xi)}=0$,
and also that the difference between the norms of the real and imaginary
parts of $\xi$ is $1$.
We can express the real and imaginary parts of $\xi$ as follows,
by taking the sum and difference of $\xi$ with its complex conjugate:
\[
\Re(\xi) = \frac{1}{2}(\xi + \cconjugate{\xi}),\qquad\Im(\xi) = \frac{1}{2\I}(\xi - \cconjugate{\xi})
\]
The inner product of two quaternions is given by equation \ref{innerprod},
but since we are dealing with pure quaternions,
the conjugates reduce to negation, and the inner product simplifies to:
\begin{equation}
\label{simpleinner}
\inner{u}{v} = -\frac{1}{2}\left(uv + vu\right)
\end{equation}
Substituting the results above for the real and imaginary parts of $\xi$,
we obtain the following:
\begin{align*}
\inner{\Re(\xi)}{\Im(\xi)}
&=-\frac{1}{2}
\left[
\frac{1}{2}\left(\xi+\cconjugate{\xi}\right)
\frac{1}{2\I}\left(\xi-\cconjugate{\xi}\right)
+
\frac{1}{2\I}\left(\xi-\cconjugate{\xi}\right)
\frac{1}{2}\left(\xi+\cconjugate{\xi}\right)
\right]\\
&=-\frac{1}{8\I}
\left[
\left(\xi+\cconjugate{\xi}\right)\left(\xi-\cconjugate{\xi}\right) +
\left(\xi-\cconjugate{\xi}\right)\left(\xi+\cconjugate{\xi}\right)
\right]\\
&=-\frac{1}{8\I}
\left[\xi^2 + \cconjugate{\xi}\xi - \cconjugate{\xi}^2 - \xi\cconjugate{\xi}
   +  \xi^2 - \cconjugate{\xi}\xi - \cconjugate{\xi}^2 + \xi\cconjugate{\xi}
\right]
=-\frac{1}{4\I}\left[\xi^2 -\cconjugate{\xi}^2\right]
\end{align*}
Since $\xi$ is a root of $-1$ by Theorem \ref{rootsofunity}, its complex conjugate must also be a
root of minus one\footnote{This is easily shown using the `complex' form $\xi = \xi_r +\I\xi_i$,
since the imaginary part must be zero in both $\xi^2$ and $\cconjugate{\xi}^2$ if they
are each equal to $-1$.}, thus
$\xi^2=\cconjugate{\xi}^2=-1$ and therefore $\inner{\Re(\xi)}{\Im(\xi)}=0$ as expected.

To show that the real and imaginary parts of $\xi$ have norms differing by $1$,
we can use a similar approach.
Starting with the fact that the norm is the inner product of a quaternion with itself,
as stated in §\,\ref{sec:seminorm},
the formula in equation \ref{simpleinner} simplifies to $\norm{u}=-u^2$ when $u$ is pure and $v=u$.
Thus the difference between the norms of the real and imaginary parts of $\xi$ is given by:
\begin{align*}
\norm{\Re(\xi)}-\norm{\Im(\xi)}=-\Re(\xi)^2+\Im(\xi)^2
&=-\left[\frac{1}{2}\left(\xi + \cconjugate{\xi}\right)\right]^2
  +\left[\frac{1}{2\I}\left(\xi - \cconjugate{\xi}\right)\right]^2\\
&=-\frac{1}{4}\left[\xi^2 + \cconjugate{\xi}^2 + \xi\cconjugate{\xi} + \cconjugate{\xi}\xi\right]
  -\frac{1}{4}\left[\xi^2 + \cconjugate{\xi}^2 - \xi\cconjugate{\xi} - \cconjugate{\xi}\xi\right]\\
&=-\frac{1}{2}\left[\xi^2 + \cconjugate{\xi}^2\right]
\end{align*}
which is $1$ because $\xi$ and its conjugate are roots of $-1$.

\subsection{Divisors of zero}
\label{sec:divisors}
It is possible for the semi-norm, $\norm{q}$, of a biquaternion to be zero,
even though $q\ne0$.
In modern terminology, the biquaternions with vanishing semi-norm
are better known as \emph{divisors of zero}.
The `rule of the norms' given in §\,\ref{sec:seminorm} shows that multiplying an
arbitrary biquaternion by a divisor of zero (with vanishing semi-norm)
yields a result which is also a divisor of zero (with vanishing semi-norm).

The conditions for the semi-norm to vanish were discovered by Hamilton
\cite[Lecture VII, §\,672, p669]{Hamilton:1853}\footnote{In
\cite[Lecture VII, §\,673, p671]{Hamilton:1853},
Hamilton rather quaintly refers to an \emph{evanescent tensor}
--- a tensor being what we here call a semi-norm.}
--- if we represent a non-zero biquaternion in the form:
$q = q_r + \I q_i$ then its norm is zero iff
$\norm{q_r} = \norm{q_i}$ and $\inner{q_r}{q_i}=0$,
that is, the real and imaginary parts must have equal norms and be perpendicular in 4-space.
This result follows easily from Lemma \ref{normlemma} in §\,\ref{sec:seminorm}.
The conditions stated are equivalent to the condition $W^2+X^2+Y^2+Z^2=0$,
which can be expanded into a pair of simultaneous equations
each equivalent to one of the two conditions above,
by equating the real and imaginary parts to zero separately
\cite[Proposition 1]{arXiv:0812:1102,10.1007/s00006-xxxx-xxxx-x}.
Notice that the conditions required for a vanishing semi-norm may be satisfied in
the case of a pure biquaternion (the real and imaginary components are then perpendicular in 3-space).
It follows from these conditions that any non-zero quaternion
or any non-zero imaginary biquaternion must have a non-zero norm,
since the constraint of real and imaginary parts having equal norm is violated.

Sangwine and Alfsmann have shown in \cite{arXiv:0812:1102,10.1007/s00006-xxxx-xxxx-x}
that every biquaternion divisor of zero is one of:
\begin{itemize}
\item an idempotent as defined in the next section in Theorem \ref{idempotents};
\item a complex multiple of an idempotent (dividing the divisor of zero by
      twice its scalar part yields an idempotent);
\item a nilpotent as defined in §\,\ref{sec:nilpotents}.
\end{itemize}

\subsubsection[Idempotents]{Idempotents}
\label{sec:idempotents}
An idempotent is a value that squares to give itself. An example of a biquaternion idempotent is
given in \cite{arXiv:physics/0508036} (quoting a paper of Lanczos) but no general case is given.
We present here a derivation of the set of idempotents in \B,
drawing on results first presented in \cite{arXiv:0812:1102,10.1007/s00006-xxxx-xxxx-x} by Sangwine and Alfsmann.

\begin{theorem}\textnormal{\cite[Theorem 2]{arXiv:0812:1102,10.1007/s00006-xxxx-xxxx-x}}\label{idempotents}
Any biquaternion of the form $q = \frac{1}{2}\pm\frac{1}{2}\xi\I$, where $\xi\in\B$ is a root of $-1$,
is an idempotent. There are no other idempotents in \B.
\end{theorem}
\begin{proof}
The proof is by construction from an arbitrary biquaternion represented in the form
$q = A + \xi B$ ($A, B \in\C$). Squaring and equating to $q$ gives:
\[
q^2 = A^2 - B^2 + 2 A B \xi = A + \xi B = q
\]
Equating coefficients of $1$ and $\xi$ (that is, equating the real and imaginary parts with
respect to $\xi$) gives:
\begin{align*}
A &= A^2 - B^2\\
B &= 2 A B
\end{align*}
The second of these equations requires that $A = \frac{1}{2}$.
Making this substitution into the first equation gives:
$\frac{1}{2} = \frac{1}{4} - B^2$, from which $B^2 = -\frac{1}{4}$. Since $B$ is complex, the
only solutions are $B = \pm\I/2$. Thus $q = \frac{1}{2}\pm\frac{\xi\I}{2}$.
\end{proof}
The roots of $-1$ in \B include two trivial cases, as noted in \cite[Theorem 2.1]{10.1007/s00006-006-0005-8}.
These are $\pm\I$ which yields the trivial idempotents $q=0$ and $q=1$; and $\pm\mu$ where $\mu$ is a
real pure quaternion root of $-1$.

Since every biquaternion idempotent, $q$, is a solution of $q^2=0$,
it must also be a solution of $q(q-1)=0$.
Therefore every biquaternion idempotent is also a divisor of zero and must
satisfy the conditions stated in §\,\ref{sec:divisors}.
This is demonstrated in detail in \cite[§\,4, Proposition 2]{arXiv:0812:1102,10.1007/s00006-xxxx-xxxx-x}.

All non-pure biquaternion divisors of zero have been shown to be of the form
$p = \alpha q$ where $\alpha$ is a non-zero complex number,
and $q$ is a biquaternion idempotent \cite[§\,4, Theorem 3]{arXiv:0812:1102,10.1007/s00006-xxxx-xxxx-x}.
Further, dividing a non-pure biquaternion divisor of zero by twice its scalar
part yields an idempotent \cite[§\,4, Corollary 1]{arXiv:0812:1102,10.1007/s00006-xxxx-xxxx-x}.

\subsubsection{Nilpotents}
\label{sec:nilpotents}
A nilpotent is a quantity whose square vanishes. Hamilton discovered nilpotent
biquaternions and was aware that all such biquaternions were divisors of zero
\cite[Lecture VII, §\,674, pp671-3]{Hamilton:1853}.

We first present the case of a pure biquaternion with vanishing semi-norm
and show a result due to Hamilton \cite[Lecture VII, §\,672, p669]{Hamilton:1853}
that the square of such a biquaternion vanishes.
\begin{theorem}\label{vanishsquare}
Let $p = p_r + \I p_i$ be a non-zero pure biquaternion with vanishing semi-norm.
Then $p^2=0$.
\end{theorem}
\begin{proof}
From the fact that $p$ has a vanishing semi-norm, $\norm{p_r} = \norm{p_i}$,
therefore we may divide by their common norm,
and obtain $p/\norm{p_r} = \mu + \I\nu$ where $\mu$ and $\nu$ are perpendicular unit pure quaternions.
Then:
\[
\left(\frac{p}{\norm{p_r}}\right)^2 = (\mu + \I\nu)^2 = \mu^2 - \nu^2 + \I(\mu\nu + \nu\mu)
\]
which vanishes because the squares of the two unit pure quaternions are $-1$
and the imaginary part vanishes because the products of two perpendicular unit pure quaternions
changes sign when the order of the product is reversed.
\end{proof}
Sangwine and Alfsmann, in \cite[§\,5, Lemma 1]{arXiv:0812:1102,10.1007/s00006-xxxx-xxxx-x} have shown that all
nilpotent biquaternions are pure,
and in \cite[§\,5, Proposition 3]{arXiv:0812:1102,10.1007/s00006-xxxx-xxxx-x} that all nilpotent biquaternions
are divisors of zero.
It follows that all nilpotent biquaternions have real and imaginary parts with equal norm
and therefore that any nilpotent biquaternion can be normalised to the form $\mu+\I\nu$ as
in Theorem \ref{vanishsquare}.
We see therefore that all nilpotent biquaternions are constructed from a pair of mutually
orthogonal unit pure quaternions.

A nilpotent biquaternion can be written in the form $\xi B$ where $\xi$ is a root of $-1$
and $B$ is complex (the square root of the semi-norm: $B = \sqrt{X^2+Y^2+Z^2}$).
Since a nilpotent biquaternion is a divisor of zero, its semi-norm is zero.
This means that $B$ is zero, even though $\xi B$ is not.
A consequence is that it is not possible to compute the `axis' of a nilpotent
biquaternion by dividing by $B$.

\subsection{The biquaternions as a geometric algebra}\label{sec:geometric}
In this section we discuss equivalence of the components of a biquaternion to the
elements of a geometric algebra
\cite{HestenesSobczyk:1984,Sommer:2001,HestenesLiRockwood:2001,10.1098/rsta.2000.0517,gaprimer:2003}
or \emph{Clifford algebra}.
Jaap Suter's tutorial paper \cite{gaprimer:2003} is particularly recommended.
A full discussion of the biquaternions as a geometric algebra is a large and non-trivial topic,
and will be left to a later paper or papers.

Hamilton himself, in his \textit{Lectures on Quaternions} \cite[Lectures I, II, III]{Hamilton:1853}
was interested in an algebra of points, lines and planes, and clearly was motivated by a desire to
have an algebra applicable to 3-dimensional geometry. The modern notion of a geometric algebra uses
different language, and some of the concepts are more sophisticated, but essentially the aim is the
same: to represent and manipulate geometric objects algebraically.

A geometric algebra contains four types of element, as already introduced in §\,\ref{sec:complexq}
and Tables~\ref{correspondence} and~\ref{geomconj}: scalars, vectors, bivectors and pseudoscalars.
Scalars are quantities without geometric form (or points),
vectors represent directed magnitudes (or lines),
bivectors represent directed areas, and pseudoscalars represent signed volumes.
The product of two perpendicular vectors is a bivector, representing the area swept
out by one vector when moved along the other \cite[§\,2.1]{gaprimer:2003}.
This is why vectors are identified with pure imaginary biquaternions.
Given two perpendicular pure quaternions, $\mu$ and $\nu$, such that $\inner{\mu}{\nu}=0$,
the two vectors $\I\mu$ and $\I\nu$ have a product $-\mu\nu$ which is perpendicular to
both $\mu$ and $\nu$.
The negative sign is a consequence of squaring \I.
Changing the order of the two vectors changes the sign of the resulting bivector
(it is oppositely directed), so the minus sign is of little consequence.
It is usual in geometric algebra to have to choose the sense by convention.
Here the rules of biquaternion multiplication define the convention for us.
The product of three orthogonal vectors defines a volume, represented by a pseudoscalar.
For example, $(\I\i)(\I\j)(\I\k)=\I$, as given by Ward \cite[p\,113]{Ward:1997}.

There has been much confusion caused by the use of the term \emph{vector}
when in fact the objects under consideration were \emph{bivectors}.
In physics the concepts of \emph{axial} and \emph{polar} vectors have been
used, causing further confusion, because the two terms suggest different
types of vector, rather than fundamentally different types of quantity.
Even the present authors, in their earlier works, used the term vector
to refer to quantities which are now clearly seen to be bivectors.
Jaap Suter in his \textit{Geometric Algebra Primer} states clearly:
\begin{quote}
A quaternion is a scalar plus a bivector.
\end{quote}
and he is right \cite[p\,50]{gaprimer:2003}.
Unfortunately, the scalar/vector part terminology commonly applied to quaternions
has this wrong, and we are stuck with it.

In geometric algebra a composite quantity consisting of two or more
of the four types of quantity (scalar, vector, bivector, pseudoscalar)
is called a \emph{multivector}.
A biquaternion may contain all four types of geometric element,
and except in degenerate cases (such as a pure imaginary biquaternion),
corresponds to the concept of a multivector.

Once we start to multiply multivectors, the geometric interpretation becomes
more difficult, although the rules of multiplication are well-understood and
easily described: the full biquaternion product can be expressed, exactly as
in the quaternion case, as:
\[
p\,q = \Scalar{p}\Scalar{q} + \Scalar{p}\Vector{q} + \Scalar{q}\Vector{p}
     + \Vector{p}\Vector{q}
\]
and the last term can be written as $\Vector{p}\cdot\Vector{q}+\Vector{p}×\Vector{q}$ or
$\inner{\Vector{p}}{\Vector{q}}+\Vector{p}×\Vector{q}$, that is the sum of a dot or
scalar or inner product, and a cross product.
A deeper analysis of the biquaternions as a geometric algebra requires further work.

Although multivectors in geometric algebra are regarded as composites of
scalar, vector, bivector and pseudoscalar,
a biquaternion is not constructed in this way.
Instead it has four complex components,
but as we have seen in §\,\ref{sec:complexq},
it is possible to interpret these in multiple ways, for example as a complex
number with quaternion components.
It is also possible to consider the biquaternions as 8-dimensional quantities with basis
$(1,\i\I,\j\I,\k\I,\i,\j,\k,\I)$ where we place the vector basis first.
The multiplication table is then easily derived from the rules of quaternion
multiplication (Table \ref{biquatmult}).
\begin{table}
\begin{center}
$
\begin{array}{r||r|rrr|rrr|r}
\hline
     &   1  & \i\I & \j\I & \k\I & \i   & \j   & \k   & \I   \\
\hline\hline
   1 &   1  & \i\I & \j\I & \k\I & \i   & \j   & \k   & \I   \\
\hline
\i\I & \i\I &   1  & -\k  &  \j  & -\I  & \k\I &-\j\I & -\i  \\
\j\I & \j\I &  \k  &   1  & -\i  &-\k\I & -\I  & \i\I & -\j  \\
\k\I & \k\I & -\j  &  \i  &   1  & \j\I &-\i\I & -\I  & -\k  \\
\hline
\i   & \i   & -\I  & \k\I &-\j\I &  -1  & \k   & -\j  & \i\I \\
\j   & \j   &-\k\I & -\I  & \i\I &  -\k &  -1  &  \i  & \j\I \\
\k   & \k   & \j\I &-\i\I & -\I  &   \j & -\i  &  -1  & \k\I \\
\hline
  \I &   \I &  -\i & -\j  & -\k  & \i\I & \j\I & \k\I &  -1  \\
\hline
\end{array}
$
\end{center}
\caption{\label{biquatmult}Biquaternion basis multiplication table.}
\end{table}

From the rules of quaternion multiplication, it is fairly easy to draw up
a table showing how the geometric components of a biquaternion multiply
together.
Table \ref{multivectormult} shows the result.
The four centre entries apply in the general case, but in the specific
cases where the bivector/vector are perpendicular, the scalar or
pseudoscalar part of the product will be zero.
Multiplication by the pseudoscalar is known as the \emph{dual} operation:
it maps a bivector into a vector, and a scalar into a pseudoscalar
or \textit{vice versa}.
\begin{table}
\begin{center}
\begin{tabular}{l@{\,}c@{\,}l||cccc}
\hline
  &  &              & S &   B   &   V   & P\\
\hline
\hline
S &--& scalar       & S &   B   &   V   & P\\
B &--& bivector     & B & S + B & P + B & V \\
V &--& vector       & V & P + B & S + B & B\\
P &--& pseudoscalar & P &   V   &   B   & S\\
\hline
\end{tabular}
\end{center}
\caption{\label{multivectormult}Multiplication table for components of a biquaternion multivector}
\end{table}

\subsubsection{The wedge or outer product}
In geometric algebra, the algebraic or geometric product of two vectors is the sum
of the scalar product and the wedge ($\wedge$) or outer product (also known as the cross
product in three dimensions):
\[
p\,q = p \cdot q + p \wedge q
\]
It is possible to define a wedge product for quaternions and biquaternions
using a formula widely used in geometric algebra \cite[Equation 3.10]{gaprimer:2003}: 
\[
p \wedge q = \frac{1}{2}\left(p\,q - q\,p\right)
\]
This gives the cross product of the vector parts (because the components of
the products involving the scalar parts, and the dot product of the vector
parts, commute and therefore cancel).
If applied to a pair of vectors (i.e. pure biquaternions with zero real part),
the result is a bivector (i.e. a pure biquaternion with zero imaginary part).
Applied to a pair of bivectors (pure quaternions) the result is another bivector.
Applied to a vector and a bivector, the result is a vector.

\section{Further work}
There is much further interpretation needed of the biquaternions as a geometric algebra.
At present formulae are lacking for vector/bivector products to yield pseudoscalars
and for generalisations of the wedge product.

The interpretation of complex angles in the polar form and in the context of the inner
product requires further work.

\bibliographystyle{hplain}
\bibliography{sangwine,quaternions,maths,clifford,online}

\begin{thebibliography}{10}

\bibitem{Altmann:1986}
Simon~L. Altmann.
\newblock {\em Rotations, Quaternions, and Double Groups}.
\newblock Oxford University Press, Oxford, 1986.

\bibitem{Artmann:88}
Benno Artmann.
\newblock {\em The concept of number : from quaternions to monads and
  topological fields}.
\newblock Ellis Horwood series in mathematics and its applications. Ellis
  Horwood, Halsted, Chichester, 1988.
\newblock Translation of: Der Zahlbeginff, Gottigen: Vandenhoeck \& Rupprecht,
  1983. Translated with additional exercises and material by H.B. Griffiths.

\bibitem{CollinsDictMaths}
E.~J. Borowski and J.~M. Borwein, editors.
\newblock {\em Collins Dictionary of Mathematics}.
\newblock HarperCollins, Glasgow, 2nd edition, 2002.

\bibitem{Bouvier:2e}
Alain Bouvier, Michel George, and François~Le Lionnais, editors.
\newblock {\em Dictionnaire des Mathématiques}.
\newblock Quadrige/Puf, Paris, 2e edition, 2005.

\bibitem{10.1112/plms/s1-4.1.381}
William~K. Clifford.
\newblock Preliminary sketch of biquaternions.
\newblock {\em Proc. London Math. Soc.}, s1-4(1):381--395, 1871,
  http://plms.oxfordjournals.org/cgi/reprint/s1-4/1/381.pdf.

\bibitem{Coxeter:1946}
H.~S.~M. Coxeter.
\newblock Quaternions and reflections.
\newblock {\em American Mathematical Monthly}, 53(3):136--146, March 1946.

\bibitem{arxiv:hep-th/9806058}
Stefano de~Leo and Waldyr~A. Rodrigues, Jr.
\newblock Quaternionic electron theory: Geometry, algebra and {D}irac's
  spinors.
\newblock Preprint: \url{http://www.arxiv.org/abs/hep-th/9806058v1}, June 1998,
  arXiv:hep-th/9806058.

\bibitem{10.1023/A:1026692508708}
Stefano de~Leo and Waldyr~A. Rodrigues, Jr.
\newblock Quaternionic electron theory: Geometry, algebra, and {D}irac's
  spinors.
\newblock {\em International Journal of Theoretical Physics}, 37(6):1707--1720,
  June 1998.

\bibitem{arXiv:math.RA/0506034}
T.~A. Ell and S.~J. Sangwine.
\newblock Quaternion involutions.
\newblock Preprint: \url{http://www.arxiv.org/abs/math.RA/0506034}, June 2005,
  arXiv:math.RA/0506034.

\bibitem{lqstfm}
T.~A. Ell and S.~J. Sangwine.
\newblock Linear quaternion systems {T}oolbox for {M}atlab\textregistered.
\newblock \url{http://lqstfm.sourceforge.net/}, 2007.
\newblock Software library, licensed under the GNU General Public License.

\bibitem{10.1016/j.camwa.2006.10.029}
T.~A. Ell and S.~J. Sangwine.
\newblock Quaternion involutions and anti-involutions.
\newblock {\em Computers and Mathematics with Applications}, 53(1):137--143,
  January 2007.

\bibitem{arXiv:physics/0508036}
Andre Gsponer and Jean-Pierre Hurni.
\newblock Lanczos -- {E}instein -- {P}etiau: {F}rom {D}irac's equation to
  nonlinear wave mechanics.
\newblock Preprint, August 2005, arXiv:physics/0508036v2.

\bibitem{Gurlebeck}
Klaus G\"{u}rlebeck and Wolfgang Spr\"{o}ssig.
\newblock {\em Quaternionic and {C}lifford Calculus for Physicists and
  Engineers}.
\newblock John Wiley, Chichester, 1997.

\bibitem{Hamiltonpapers:V3}
H.~Halberstam and R.~E. Ingram, editors.
\newblock {\em The Mathematical Papers of Sir William Rowan Hamilton}, volume
  III Algebra.
\newblock Cambridge University Press, Cambridge, 1967.

\bibitem{Hamilton:1844}
W.~R. Hamilton.
\newblock On a new species of imaginary quantities connected with the theory of
  quaternions.
\newblock {\em Proceedings of the Royal Irish Academy}, 2:424--434, 1844.

\bibitem{Hamilton:1848}
W.~R. Hamilton.
\newblock Researches respecting quaternions.
\newblock {\em Transactions of the Royal Irish Academy}, 21:199--296, 1848.

\bibitem{Hamilton:1853}
W.~R. Hamilton.
\newblock {\em Lectures on Quaternions}.
\newblock Hodges and Smith, Dublin, 1853.
\newblock Available online at Cornell University Library:
  \url{http://historical.library.cornell.edu/math/}.

\bibitem{Hamilton:1866}
W.~R. Hamilton.
\newblock {\em Elements of Quaternions}.
\newblock Longmans, Green and Co., London, 1866.

\bibitem{Hamiltonpapers:V3:5}
W.~R. Hamilton.
\newblock On a new species of imaginary quantities connected with the theory of
  quaternions.
\newblock In Halberstam and Ingram \cite{Hamiltonpapers:V3}, chapter~5, pages
  111--116.
\newblock First published as \cite{Hamilton:1844}.

\bibitem{Hamiltonpapers:V3:8}
W.~R. Hamilton.
\newblock On quaternions.
\newblock In Halberstam and Ingram \cite{Hamiltonpapers:V3}, chapter~8, pages
  227--297.
\newblock First published in various articles in \textit{Philosophical
  Magazine}, 1844--1850.

\bibitem{Hamiltonpapers:V3:35}
W.~R. Hamilton.
\newblock On the geometrical interpretation of some results obtained by
  calculation with biquaternions.
\newblock In Halberstam and Ingram \cite{Hamiltonpapers:V3}, chapter~35, pages
  424--5.
\newblock First published in \textit{Proceedings of the Royal Irish Academy},
  1853.

\bibitem{Hamiltonpapers:V3:7}
W.~R. Hamilton.
\newblock Researches respecting quaternions. {F}irst series (1843).
\newblock In Halberstam and Ingram \cite{Hamiltonpapers:V3}, chapter~7, pages
  159--226.
\newblock First published as \cite{Hamilton:1848}.

\bibitem{HestenesLiRockwood:2001}
David Hestenes, Hongbo Li, and Alyn Rockwood.
\newblock New algebraic tools for classical geometry.
\newblock In Sommer \cite{Sommer:2001}, chapter~1, pages 3--26.

\bibitem{HestenesSobczyk:1984}
David Hestenes and Garret Sobczyk.
\newblock {\em Clifford Algebra to Geometric Calculus}.
\newblock D. Reidel Publishing Company, Dordrecht, 1984.

\bibitem{Hitzer:TR_2009_3}
Eckhard Hitzer and Rafa{\l} Ab{\l}amowicz.
\newblock Geometric roots of $-1$ in {C}lifford algebras ${C}\ell_{p,q}$ with
  $p+q\leq4$.
\newblock Technical Report 2009--3, Tennessee Technological University,
  Department of Mathematics, Cookeville, TN~38505, USA, May 2009.

\bibitem{arXiv:0905.3019}
Eckhard Hitzer and Rafa{\l} Ab{\l}amowicz.
\newblock Geometric roots of $-1$ in {C}lifford algebras ${C}\ell_{p,q}$ with
  $p+q\leq4$.
\newblock Preprint: \url{http://arxiv.org/abs/arxiv:0905.3019}, May 2009,
  arXiv:0905.3019.

\bibitem{Kantor:1989}
I.~L. Kantor and A.~S. Solodnikov.
\newblock {\em Hypercomplex numbers, an elementary introduction to algebras}.
\newblock Springer-Verlag, New York, 1989.

\bibitem{KellandTait}
Philip Kelland and Peter~Guthrie Tait.
\newblock {\em Introduction to quaternions}.
\newblock Macmillan, London, 3rd edition, 1904.

\bibitem{Kuipers:1999}
J.~B. Kuipers.
\newblock {\em Quaternions and Rotation Sequences}.
\newblock Princeton University Press, Princeton, New Jersey, 1999.

\bibitem{10.1098/rsta.2000.0517}
Joan Lasenby, Anthony~N. Lasenby, and Chris J.~L. Doran.
\newblock A unified mathematical language for physics and engineering in the
  21st century.
\newblock {\em Philosophical Transactions of the Royal Society A: Mathematical,
  Physical and Engineering Sciences}, 358:21--39, 2000.

\bibitem{PenguinDictMaths3e}
David Nelson, editor.
\newblock {\em The Penguin Dictionary of Mathematics}.
\newblock Penguin Books, London, third edition, 2003.

\bibitem{Porteous:1981}
I.~R. Porteous.
\newblock {\em Topological Geometry}.
\newblock Cambridge University Press, Cambridge, second edition, 1981.

\bibitem{arXiv:math.RA/0506190}
S.~J. Sangwine.
\newblock Biquaternion (complexified quaternion) roots of -1.
\newblock Preprint: \url{http://arxiv.org/abs/math.RA/0506190}, June 2005,
  arXiv:math.RA/0506190.

\bibitem{10.1007/s00006-006-0005-8}
S.~J. Sangwine.
\newblock Biquaternion (complexified quaternion) roots of -1.
\newblock {\em Advances in Applied Clifford Algebras}, 16(1):63--68, February
  2006.
\newblock Publisher: Birkhäuser Basel.

\bibitem{arXiv:0812:1102}
S.~J. Sangwine and Daniel Alfsmann.
\newblock Determination of the biquaternion divisors of zero, including the
  idempotents and nilpotents.
\newblock Preprint: \url{http://arxiv.org/abs/arxiv:0812.1166}, December 2008,
  arXiv:0812.1102.

\bibitem{10.1007/s00006-xxxx-xxxx-x}
S.~J. Sangwine and Daniel Alfsmann.
\newblock Determination of the biquaternion divisors of zero, including the
  idempotents and nilpotents.
\newblock {\em Advances in Applied Clifford Algebras}, 10pp, December
  2009.
\newblock Publisher: Birkhäuser Basel, in press.

\bibitem{qtfm}
S.~J. Sangwine and N.~{Le Bihan}.
\newblock Quaternion {T}oolbox for {M}atlab\textregistered.
\newblock \url{http://qtfm.sourceforge.net/}, 2005.
\newblock Software library, licensed under the GNU General Public License.

\bibitem{10.1007/s00006-008-0128-1}
S.~J. Sangwine and N.~{Le~Bihan}.
\newblock Quaternion polar representation with a complex modulus and complex
  argument inspired by the {C}ayley-{D}ickson form.
\newblock {\em Advances in Applied Clifford Algebras}, 10pp, 2008.
\newblock Publisher: Birkhäuser Basel, in press.

\bibitem{arXiv:0802.0852}
S.~J. Sangwine and N.~{Le~Bihan}.
\newblock Quaternion polar representation with a complex modulus and complex
  argument inspired by the {C}ayley-{D}ickson form.
\newblock Preprint: \url{http://arxiv.org/abs/arxiv:0802.0852}, February 2008,
  arXiv:0802.0852.

\bibitem{Sommer:2001}
G.~Sommer, editor.
\newblock {\em Geometric computing with Clifford algebras: theoretical
  foundations and applications in computer vision and robotics}.
\newblock Springer-Verlag, London, UK, 2001.

\bibitem{Sudbery:1979}
A.~Sudbery.
\newblock Quaternionic analysis.
\newblock {\em Mathematical Proceedings of the Cambridge Philosophical
  Society}, 85(2):199--225, March 1979.

\bibitem{gaprimer:2003}
Jaap Suter.
\newblock Geometric algebra primer.
\newblock Self-published on personal website:
  \url{http://www.jaapsuter.com/paper/ga_primer.pdf}, March 2003.

\bibitem{Synge:1972}
J.~L. Synge.
\newblock Quaternions, {L}orentz transformations, and the
  {C}onway-{D}irac-{E}ddington matrices.
\newblock Communications of the Dublin Institute for Advanced Studies, Series
  A~21, Dublin Institute for Advanced Studies, Dublin, 1972.

\bibitem{Cayley:1890}
P.~G. Tait.
\newblock Sketch of the analytical theory of quaternions.
\newblock In {\em An elementary treatise on Quaternions}, chapter~VI, pages
  146--159. Cambridge University Press, third edition, 1890.
\newblock Chapter by `Prof Cayley' (Arthur Cayley).

\bibitem{Ward:1997}
J.~P. Ward.
\newblock {\em Quaternions and Cayley Numbers: Algebra and Applications},
  volume 403 of {\em Mathematics and Its Applications}.
\newblock Kluwer, Dordrecht, 1997.

\end{thebibliography}
\end{document}